\documentclass{amsart}
\usepackage{amssymb}
\usepackage{amsbsy}
\usepackage{epstopdf}
\usepackage{float}
\usepackage{amsfonts}
\usepackage{amsmath}
\usepackage{amscd}
\usepackage{tikz-cd}
\usepackage{latexsym}
\usepackage{bbm}
\usepackage{verbatim}
\usepackage{paralist}
\usepackage[utf8]{inputenc}
\usepackage{etoolbox}
\usepackage{enumitem}
\usepackage{xspace}
\usetikzlibrary{patterns}
\usepackage{hyperref}

\newtheorem{theorem}{Theorem}[section]

\newtheorem{lemma}[theorem]{Lemma}
\newtheorem{observation}[theorem]{Observation}
\newtheorem{proposition}[theorem]{Proposition}

\newtheorem{fact}[theorem]{Fact}
\newtheorem*{theorem*}{Theorem} 

\theoremstyle{definition}
\newtheorem{definition}[theorem]{Definition}
\newtheorem{question}[theorem]{Question}

\theoremstyle{remark}

\newcounter{claimcounter}
\numberwithin{claimcounter}{theorem}
\newtheorem{claimno}[claimcounter]{Claim}

\setlist[description]{leftmargin=5.5em,labelindent=\parindent}

\renewcommand{\P}{\mathbb{P}}
\newcommand{\Q}{\mathbb{Q}}

\newcommand{\R}{\mathbb{R}}

\renewcommand{\c}{\mathfrak{c}}

\newcommand{\PFA}{\textup{\ensuremath{\textsf{PFA}}}}
\newcommand{\MA}{\textup{\ensuremath{\textsf{MA}}}}

\newcommand{\ZFC}{\textup{\ensuremath{\textsf{ZFC}}}}

\newcommand{\BSCFA}{\textup{\ensuremath{\textsf{BSCFA}}}}
\newcommand{\BPFA}{\textup{\ensuremath{\textsf{BPFA}}}}
\newcommand{\CH}{\textup{\textsf{CH}}}

\newcommand{\MP}{\textup{\ensuremath{\textsf{MP}}}}
\newcommand{\bfMP}{\textup{\ensuremath{\textsf{\textbf{MP}}}}}
\newcommand{\RA}{\textup{\ensuremath{\textsf{RA}}}}
\newcommand{\bfRA}{\textup{\ensuremath{\textsf{\textbf{RA}}}}}
\newcommand{\HC}{\textup{\ensuremath{H_{\omega_1}}}}

\DeclareMathOperator{\Coll}{\mathcal C\textit{oll}\,}

\newcommand{\st}{\; | \;}
\newcommand{\set}[2]{\left\{#1\st #2 \right\}}
\newcommand{\seq}[2]{\langle #1 \st #2 \rangle}

\newcommand{\Ptail}{\P_{\mathrm{tail}}}
\newcommand{\Gtail}{G_{\mathrm{tail}}}

\newcommand{\forces}{\Vdash}
\newcommand{\proves}{\vdash}

\DeclareMathOperator{\MPsc}{ \textup{\textsf{MP}}_{\textit{sc}}}
\newcommand{\bfMPsc}{\textbf{\textup{\textsf{MP}}}_{\textit{sc}}}
\DeclareMathOperator{\MPc}{\textup{\textsf{MP}}_{\textit{c}}}
\newcommand{\bfMPc}{\textbf{\textup{\textsf{MP}}}_{\textit{c}}}
\DeclareMathOperator{\MPp}{\textup{\textsf{MP}}_{\textit{p}}}
\newcommand{\bfMPp}{\textbf{\textup{\textsf{MP}}}_{\textit{p}}}
\DeclareMathOperator{\MPccc}{\textup{\textsf{MP}}_{\textit{ccc}}}
\newcommand{\bfMPccc}{\textbf{\textup{\textsf{MP}}}_{\textit{ccc}}}

\newcommand{\bflMPsc}{ \textup{\textsf{\textbf{LMP}}}_{\textit{sc}} }
\newcommand{\bflMPc}{ \textup{\textsf{\textbf{LMP}}}_{\textit{c}} }
\newcommand{\bflMPccc}{ \textup{\textsf{\textbf{LMP}}}_{\textit{ccc}} }
\newcommand{\bflMPp}{ \textup{\textsf{\textbf{LMP}}}_{\textit{p}} }

\newcommand{\lMP}[2]{ \textup{\textsf{MP}}^{H_{#2}}_{#1}(H_{#2})}
\newcommand{\bflMP}{ \textup{\textsf{\textbf{LMP}}} }

\newcommand{\BFA}[2]{\textup{\textsf{BFA}}_{#1}({#2})}

\DeclareMathOperator{\RAsc}{\textup{\textsf{RA}}_{\textit{sc}}}
\DeclareMathOperator{\RAc}{\textup{\textsf{RA}}_{\textit{c}}}
\DeclareMathOperator{\RAp}{\textup{\textsf{RA}}_{\textit{p}}}
\DeclareMathOperator{\RAccc}{\textup{\textsf{RA}}_{\textit{ccc}}}

\DeclareMathOperator{\bfRAsc}{\textup{\textsf{\textbf{RA}}}_{\textit{sc}}}
\DeclareMathOperator{\bfRAc}{\textup{\textsf{\textbf{RA}}}_{\textit{c}}}
\DeclareMathOperator{\bfRAccc}{\textup{\textsf{\textbf{RA}}}_{\textit{ccc}}}
\DeclareMathOperator{\bfRAp}{\textup{\textsf{\textbf{RA}}}_{\textit{p}}}

\newcommand{\TC}[1]{\mathrm{TC}(\{ #1 \})}

\begin{document}
\title{Combining Resurrection and Maximality}
\author{Kaethe Minden}
 \address[K.~Minden]{Bard College at Simon's Rock \\ 84 Alford Road \\ Great Barrington, MA 01230}
\email{kminden@simons-rock.edu}
\urladdr{https://kaetheminden.wordpress.com/}
\date{}     					
\thanks{Some of the material presented here is based on the author's doctoral thesis \cite{Minden:2017fr}, written at the Graduate Center of CUNY under the supervision of Gunter Fuchs. The author would like to thank the referee for their suggestions and for pointing out key oversights in the initial draft.}


\begin{abstract}
It is shown that the resurrection axiom and the maximality principle may be consistently combined for various iterable forcing classes. The extent to which resurrection and maximality overlap is explored via the local maximality principle.
\end{abstract}
\maketitle

\section{Introduction}
The maximality principle (\textsf{MP}) was originally defined by Stavi and V\"a\"an\"anen \cite{Stavi:2001qv} for the class of $ccc$ forcings. Maximality principles were defined in full generality by Hamkins~\cite{Hamkins:2003jk}, and expanded upon for different classes of forcing notions by Fuchs~\cite{Fuchs:2008rt, Fuchs:2008ve},  and Leibman~\cite{Leibman:MP}.
The axiom $\MP$ states that if a sentence may be forced in such a way that it remains true in every further forcing extension, then it must have been true already, in the original ground model. The resurrection axiom ($\RA$) is due to Hamkins and Johnstone \cite{Hamkins:2013qv}. Very roughly speaking, $\RA$ posits that no matter how you force, it is always possible to force again to ``resurrect" the validity of certain sentences - meaning a statement may not be true after some forcing, but there is always a further forcing which will undo this harm. In fact the axiom grants a bit more than this, and posits an amount of elementarity between the two-step extension and the ground model.

A forcing class $\Gamma$ is generally meant to be definable, closed under two-step iterations, and to contain trivial forcing. 
Resurrection and maximality may hold for more forcing classes than forcing axioms can, while they tend to imply their relevant bounded forcing axiom counterparts. In order for a forcing axiom to make sense for a particular class of forcing notions, the forcings should preserve stationary subsets of $\omega_1$. However, this restriction does not exist for the resurrection axiom or the maximality principle.

A question reasonable to ask about any forcing class is whether or not the resurrection axiom and the maximality principle may consistently both hold for that class. I answer the question positively and show that the consistency strength of the combined boldface principles together is that of a strongly uplifting fully reflecting cardinal. 

Perhaps many forcing classes are ignored in this paper. The focus is on forcing classes containing forcing notions which may potentially either collapse cardinals (to $\aleph_1$ or larger), add reals, or both. I also want to look at forcing classes which have a corresponding forcing axiom. Thus the focus is on proper, $ccc$, countably closed, subcomplete\footnote{See~\cite{Jensen:2014} and \cite{Minden:2017fr} for the definition of subcomplete forcing and its properties.}, and the class of all forcing notions. The results stated here for proper forcing should work similarly for semiproper and subproper forcing. Future work may certainly flesh out important distinctions.

In section \ref{sec:MP} the definition of the maximality principle is given for various classes of forcing, and the relevant equiconsistency result is stated. The same is done for the resurrection axiom in section \ref{sec:RA}. In section \ref{sec:RA+MP}, it is shown that the two may consistently be combined. In \ref{sec:LMP}, the local maximality principle is introduced as a natural axiom similar to a bounded forcing axiom but stronger, which both the maximality principle and the resurrection axiom imply. The consistency of local maximality is established at the end of the section.

\section{The Maximality Principle}
\label{sec:MP}
One motivation behind maximality principles is their connection to modal logic. In modal logic, necessary ($\Box$) and possible ($\lozenge$) are modal operators. In our context we interpret ``possible" as forceable, or true in some forcing extension, and ``necessary" as true in every forcing extension.

\begin{definition}
Let $\Gamma$ be a forcing class defined by a formula (to be evaluated in the forcing extensions in cascaded modal operator uses). 

We say that a sentence $\varphi(\vec a)$ is $\Gamma$-\emph{forceable} if there is $\P \in \Gamma$ such that for every $q \in \P$, we have that $q \forces \varphi(\vec a)$. In other words, a statement is $\Gamma$-forceable if it is forced to be true in an extension by a forcing from $\Gamma$. 

A sentence $\varphi(\vec a)$ is $\Gamma$-\emph{necessary} if for all $\P \in \Gamma$ and all $q \in \P$, we have that $q \forces \varphi(\vec a)$. So a sentence is $\Gamma$-necessary if it holds in any forcing extension by a forcing notion from $\Gamma$. If $\Gamma$ contains the trivial forcing then a statement being $\Gamma$-necessary implies that it is true. 

If $S$ is a term in the language of set theory, then the \emph{Maximality Principle} for $\Gamma$ with parameters from $S$, which we denote $\textsf{MP}_{\Gamma}(S)$, is the scheme of formulae stating that every sentence with parameters $\vec a$ from $S$ that is $\Gamma$-forceably $\Gamma$-necessary is true. I.e., if the sentence ``$\varphi(\vec a)$ is $\Gamma$-necessary" is $\Gamma$-forceable, then $\varphi(\vec a)$ is true. In brief, the maximality principle posits $\lozenge \Box \varphi \implies \varphi$.
\end{definition}

Write $\MPsc$ to stand for $\textsf{MP}_\Gamma$ where $\Gamma=\set{ \P }{ \P \text{ is subcomplete} }$, $\MPc$ in the case where $\Gamma$ is countably closed forcings, $\MPp$ for the class of proper forcings, and $\MPccc$ for the class of $ccc$ forcings. We leave out $\Gamma$ if we are considering all forcing notions. 
Since all of these classes of forcing notions $\Gamma$ are closed under two-step iterations and contain trivial forcing, it follows that $\textsf{MP}_\Gamma$ is equivalent to the statement that every sentence that is $\Gamma$-forceably $\Gamma$-necessary is $\Gamma$-necessary. 
 
Before moving on, it should be noted that it of course does not technically make sense to write $\textsf{MP}_\Gamma \implies P$ for some proposition $P$ in the language of set theory, since $\MP_{\Gamma}$ is a scheme. When something like this is written, it should be interpreted as saying instead $\ZFC+\textsf{MP}_\Gamma \proves P$.

First we analyze the parameter set $S$ that may be allowed in the definition. Since being hereditarily countable is forceably necessary  (\cite[Obs.~3]{Hamkins:2003jk}), it is clear that $S = \HC$ is the natural parameter set for the maximality principle for the class of all forcing notions. We will write $\bfMP$ for the boldface version of the maximality principle for all forcing, i.e.,  $\bfMP=\MP(H_{\omega_1})$. For the case where we consider $\MPc$ as in \cite{Fuchs:2008rt}, the natural parameter set to use is $H_{\omega_2}$. The same is true for $\MPsc$.
The next lemma follows Fuchs \cite[Thm.~2.6]{Fuchs:2008rt}.

\begin{lemma} \label{lemma:MPscparams}
Let $\Gamma$ be a forcing class containing forcing notions which collapse arbitrarily large cardinals to $\omega_1$. Then $\MP_{\Gamma}$ cannot be consistently strengthened by allowing parameters that aren't in $H_{\omega_2}$. In particular, $$\MP_\Gamma(S) \implies S \subseteq H_{\omega_2}.$$\end{lemma}
\begin{proof}
The point is that for any set $a$, it is $\Gamma$-forceably $\Gamma$-necessary that $a \in H_{\omega_2}$. Indeed, after forcing to collapse $|\TC{a}|$ to $\omega_1$, we have that $a \in H_{\omega_2}$ in the forcing extension. This must remain true in every further forcing extension. So, if $\MPsc(\{a\})$ holds, it follows that $a \in H_{\omega_2}$.
\end{proof}

Write $\bfMPsc$ for $\MPsc(H_{\omega_2})$, $\bfMPc$ for $\MPc(H_{\omega_2})$.

\begin{lemma} \label{lemma:MPcccparams} Let $\Gamma$ be a forcing class which may add an arbitrary amount of reals but cannot collapse cardinalities. Then $$\MP_\Gamma(S) \implies S \subseteq H_\c.$$\end{lemma}
\begin{proof}
The point is that ``$2^\omega$ is greater than the hereditary size of $a$" is $\Gamma$-forceably $\Gamma$-necessary.
\end{proof}

Thus write $\bfMPccc$ for $\MPccc(H_\c)$. Note that by Lemma \ref{lemma:MPscparams}, $\MPp(H_\c) \implies H_{\c} \subseteq H_{\omega_2}$. In this paper, we choose to have the boldface version of the maximality principle for proper forcing to be $\bfMPp = \MPp(H_{\omega_2})$.

Assuming there is a regular cardinal $\delta$ satisfying $V_\delta \prec V$, the maximality principle is consistent. The proof of this uses a technique that adapts arguments of Hamkins \cite{Hamkins:2003jk}. Hamkins has described the proof as ``running through the house and turning on all the lights", in the sense that the posets that are forced are those that push ``buttons", sentences that can be ``switched on" and stay on, in all forcing extensions. A \emph{button} in our case is a sentence $\varphi(\vec a)$ that is $\Gamma$-forceably $\Gamma$-necessary.
As Hamkins discusses in detail, the existence of a regular cardinal $\delta$ such that $V_\delta \prec V$ is a scheme of formulas sometimes referred to as the ``L\'evy scheme." We refer to the L\'evy scheme as positing the existence of what we refer to as a \emph{fully reflecting} cardinal.

\begin{theorem}[{\cite[Thm.~31]{Hamkins:2003jk}}] \label{thm:fullyreflecting=MP}
The following consistency results hold.
\begin{enumerate}
	\item $\bfMP \implies \aleph_1^V$ is fully reflecting in $L$.
	\item $\bfMPccc \implies \c^V$ is fully reflecting in $L$.
	\item $\bfMPp \implies \aleph_2^V$ is fully reflecting in $L$.
	\item $\bfMPc \implies \aleph_2^V$ is fully reflecting in $L$.
	\item {\cite[Lemma 4.1.4]{Minden:2017fr}} $\bfMPsc \implies \aleph_2^V$ is fully reflecting in $L$.
\end{enumerate}
\end{theorem}

It follows that for all of these classes $\Gamma$, $\bfMP_\Gamma$ cannot hold in $L$, since otherwise $L$ would think that $\aleph_2$ or $\c$ is inaccessible in $L$, a contradiction. Of course, $\bfMP_\Gamma$ fails in $L$ whenever $\Gamma$ has any nontrivial forcing in it, since then $``V\neq L"$ is $\Gamma$-forceably $\Gamma$-necessary.

\begin{theorem}[{\cite[Thm.~32]{Hamkins:2003jk}}] \label{thm:MP=fullyreflecting} Let $\delta$ be a fully reflecting cardinal. Then there are forcing extensions in which the following hold:
\begin{enumerate}
	\item $\bfMP$ and $\delta=\c=\aleph_1$.
	\item $\bfMPccc$ and $\delta=\c$.
	\item $\bfMPp$ and $\delta=\c=\aleph_2$.
	\item $\bfMPc$ and $\delta=\aleph_2$ and $\CH$.
	\item {\cite[Thm.~4.1.3]{Minden:2017fr}} $\bfMPsc$ and $\delta = \aleph_2$ and $\CH$.
\end{enumerate}	
\end{theorem}

Here a least-counterexample lottery sum iteration, which takes lottery sums at each stage, is favored. Hamkins' proofs make use of a book-keeping function. Overall, the two methods follow the same ``running through the house" method. We define the lottery sum poset below.

\begin{definition} For a family $\mathcal P$ of forcing notions, the \emph{lottery sum} poset is defined as follows: 
	$$\bigoplus \mathcal P=\{ \mathbbm 1_{\mathcal P} \} \cup \set{ \langle \P, p \rangle }{ \P \in \mathcal P \ \land \ p \in \P }$$
with $\mathbbm 1_{\mathcal P}$ weaker than everything and $\langle \P, p \rangle \leq \langle \P',p' \rangle$ if and only if $\P = \P'$ and $p \leq_\P p'$. \end{definition}
For two forcing notions, $\P$ and $\Q$, write $\P \oplus \Q$ for $\bigoplus\{\P, \Q\}$.

One major difference encountered using a lottery sum is that $ccc$ forcing notions are not closed under lottery sums. So in the case of the $ccc$ forcing class, we use Hamkins' method more directly.

%
%
%

\section{The Resurrection Axiom}
\label{sec:RA}
The idea behind the resurrection axiom is to look at the model-theoretic concept of existential closure in the realm of forcing, because, as is pointed out by Hamkins and Johnstone (\cite{Hamkins:2013qv}), the notions of resurrection and existential closure are tightly connected in model theory. A submodel $\mathcal M \subseteq \mathcal N$ is \textit{existentially closed} in $\mathcal N$ if existential statements in $\mathcal N$ using parameters from $\mathcal M$ are already true in $\mathcal M$, i.e., $\mathcal M$ is a $\Sigma_1$-elementary substructure of $\mathcal N$. Many forcing axioms can be expressed informally by stating that the universe is existentially closed in its forcing extensions, since forcing axioms posit that generic filters, which normally exist in a forcing extension, exist already in the ground model.  Hamkins and Johnstone consider resurrection for forcing extensions to be a more ``robust" formulation of forcing axioms for various forcing classes. Resurrection axioms imply the truth of their associated bounded forcing axiom, but not the other way around. 

\begin{definition} Let $\Gamma$ be a class of forcing notions closed under two-step iterations. Let $\tau$ be a term for a cardinal to be computed in various models; e.g. $\c$, $\aleph_1$, etc. The \emph{Resurrection Axiom} $\textsf{RA}_\Gamma(H_\tau)$ asserts that for every forcing notion $\Q \in \Gamma$ there is a further forcing $\dot{\R}$ with $\forces_{\Q} \dot{\R}  \in \Gamma$ such that if $g*h \subseteq \Q * \dot{\R}$ is $V$-generic, then $H_{\tau}^V \prec H_{\tau}^{V[g*h]}.$
\end{definition}

Hamkins and Johnstone  \cite{Hamkins:2013qv} examined $\textsf{RA}_\Gamma(H_\c)$ for $\Gamma$ such as  proper, $ccc$, countably closed, and the class of all forcing notions.
The reason $H_\c$ is required in general is that if some forcing notion in $\Gamma$ adds new reals, then $H_\kappa$, where $\kappa>\c$ in $V$, simply cannot be existentially closed in the forcing extension; the added real itself is witnessing the lack of existential closure. So certainly for $\kappa>\c$ and any class of forcing notions $\Gamma$ which potentially add new reals, $\textsf{RA}_\Gamma(H_\kappa)$ cannot hold. 
For proper forcing and $ccc$ forcing, we write $\RAp$ for $\RAp(H_\c)$ and $\RAccc$ for $\RAccc(H_\c)$.
However, by the following result we see that $\RAc(H_\c)$ and $\RAsc(H_\c)$ are both equivalent to $\CH$. 

\begin{proposition}{\cite[Thm.~8]{Hamkins:2013qv}} \label{proposition:CHiffRAscHc} Suppose $\Gamma$ contains a forcing which forces $\CH$ but no forcing in $\Gamma$ adds new reals. Then $\CH \iff \RA_\Gamma(H_\c)$. \end{proposition}
\begin{proof}
For the forward implication, suppose that $\CH$ holds. Then since no new reals are added, $H_{\omega_1}$ is unaffected by each forcing in $\Gamma$, and moreover, $\c$ remains $\omega_1$ in every extension by a forcing in $\Gamma$. 

For the backward direction, assume $\RA_\Gamma(H_\c)$ holds. Let $\P$ force $\CH$. Then there is a further forcing $\dot \R$ satisfying $\forces_\P \dot \R \in \dot \Gamma$ such that letting $g *h \subseteq \P * \dot \R$ be generic, we have $H_\c \prec H_\c^{V[g*h]}.$ We know that $\CH$ has to hold still in $V[g*h]$ since no new reals are added to make $\c$ larger. Thus $\CH$ holds in $V$ by elementarity, as desired.
Indeed, $\CH$ is equivalent to the statement that $H_{\c}$ contains only one infinite cardinal, which can be expressed in $H_{\c}$. 
\end{proof}
Perhaps $\RAsc(H_\c)$, or indeed $\RAc(H_\c)$, is not necessarily the right axiom to look at. So what is the correct axiom to examine? I will discuss two reasonable possibilities for the hereditary sets: $H_{\omega_2}$ and $H_{2^{\omega_1}}$. Let's see what $\RAsc(H_{2^{\omega_1}})$ and $\RAc(H_{2^{\omega_1}})$ imply about the size of $2^{\omega_1}$.

\begin{proposition} Suppose $\Gamma$ contains forcing to collapse to $\omega_1$. Then $\RA_\Gamma(H_{2^{\omega_1}}) \implies 2^{\omega_1} = \omega_2$. 
\end{proposition}
\begin{proof}
We show the contrapositive. Let $2^{\omega_1} \geq \omega_3$. Let $\kappa = \omega_2^V$. Then $H_{2^{\omega_1}} \models ``\kappa = \omega_2"$. But after forcing to collapse $\kappa$ to $\omega_1$ we have that $\kappa < \omega_2$ in the extension. Moreover, if $\R$ is any further  forcing in $\Gamma$, we will still have that for $h \subseteq \R$ generic, $H_{2^{\omega_1}}^{V[g][h]} \models ``\kappa < \omega_2"$. So $\RA_\Gamma(H_{2^{\omega_1}})$ must fail.
\end{proof}

The next proposition gives a relationship between $\RAc(H_{2^{\omega_1}})$ and $\RAc(H_{\omega_2})$ as well as between $\RAsc(H_{2^{\omega_1}})$ and $\RAsc(H_{\omega_2})$. However, the answer to the following question is unknown.

\begin{proposition} Let $\Gamma$ be a forcing class containing forcing notions which collapse arbitrarily large cardinals to $\omega_2$ and to $\omega_1$.  Then $$\RA_\Gamma(H_{2^{\omega_1}}) \iff 2^{\omega_1}=\omega_2 + \RA_\Gamma(H_{\omega_2}).$$
\end{proposition}
\begin{proof}
For the forward direction, we already have that $\RA_\Gamma(H_{2^{\omega_1}}) \implies 2^{\omega_1}=\omega_2$ by the previous proposition. Moreover, if $\RA_\Gamma(H_{2^{\omega_1}})$ holds, so does $\RA_\Gamma(H_{\omega_2})$, since $H_{\omega_2}=H_{2^{\omega_1}}$ in the extension by elementarity.

For the backward direction, suppose that $\RA_\Gamma(H_{\omega_2})$ holds and $2^{\omega_1}=\omega_2$. We would like to show that $\RA_\Gamma(H_{2^{\omega_1}})$ holds.
Toward that end, suppose that $\Q$ is in $\Gamma$ and let $g \subseteq \Q$ be generic. Then we have that there is some forcing $\R$ with $h \subseteq \R$ generic over $V[g]$, such that 
	$H_{2^{\omega_1}}^V=H_{\omega_2}^V \prec H_{\omega_2}^{V[g*h]}.$
So if in $V[g*h]$ we have that $2^{\omega_1}=\omega_2$, then we are done. If not, let $G$ collapse $2^{\omega_1}$ to be $\omega_2$ over $V[g][h]$. Then $H_{\omega_2}^{V[g*h]}=H_{\omega_2}^{V[g*h*G]}=H_{2^{\omega_1}}^{V[g*h*G]}$, so we are done.
\end{proof}

\begin{question} Is it the case that $\RAsc(H_{2^{\omega_1}}) \iff \RAsc(H_{\omega_2})$? 
Indeed, is it the case that $\RAc(H_{2^{\omega_1}}) \iff \RAc(H_{\omega_2})$?\end{question}

In comparison to Proposition \ref{proposition:CHiffRAscHc}, the lack of any obvious restraints for the size of $2^{\omega_1}$ lends credibility to $H_{\omega_2}$ being the right parameter set to consider for forcing notions which do not add reals, but do contain the relevant collapses. So this is what we will be using. Write $\RAsc$ for $\RAsc(H_{\omega_2})$ and $\RAc$ for $\RAc(H_{\omega_2})$. 

Hamkins and Johnstone \cite{Hamkins:2013qv} determined that the consistency strength of resurrection is an uplifting cardinal. 

\begin{definition}
We say that $\kappa$ is \emph{uplifting} so long as $\kappa$ is $\theta$-uplifting for every ordinal $\theta$. This means that there is an inaccessible cardinal $\gamma \geq \theta$ such that $V_\kappa \prec V_\gamma$.
\end{definition}

Since we are focusing our attention to the boldface maximality principle in this paper, we should also look at the boldface version of resurrection, in which we carry around a kind of parameter set.

\begin{definition} Let $\Gamma$ be a fixed, definable class of forcing notions. Let $\tau$ be a term for a cardinality to be computed in various models; e.g. $\c$, $\omega_1$, etc. The \emph{Boldface Resurrection Axiom} $\textsf{\textbf{RA}}_\Gamma(H_\tau)$ asserts that for every forcing notion $\Q \in \Gamma$ and $A \subseteq H_\tau$ there is a further forcing $\dot{\R}$ with $\forces_{\Q} \dot{\R}  \in \Gamma$ such that if $g*h \subseteq \Q * \dot{\R}$ is $V$-generic, then there is an $A^* \in V[g*h]$ such that $\langle H_{\tau}^V, \in, A \rangle \prec \langle H_{\tau}^{V[g*h]}, \in, A^* \rangle. $ \end{definition}

Again, it doesn't make too much sense to talk about the resurrection axiom at the continuum for forcing classes which can't add new reals.
Thus as in the lightface versions, the notion we will be looking at for the boldface version is $\bfRAsc(H_{\omega_2})$ which we will just refer to as $\bfRAsc$, $\bfRAp$ for $\bfRAp(H_\c)$, $\bfRAccc$ for $\bfRAccc(H_\c)$, and $\bfRA$ for $\bfRA(H_\c)$. Hamkins and Johnstone \cite{Hamkins:2014yu} determined that the consistency strength of boldface resurrection is a strongly uplifting cardinal.

\begin{definition} We say that $\kappa$ is \emph{strongly uplifting} so long as $\kappa$ is $\theta$-strongly uplifting for every ordinal $\theta$. This means that for every $A \subseteq V_\kappa$ there is an inaccessible cardinal $\gamma \geq \theta$ and a set $A^* \subseteq V_\gamma$ such that $\langle V_\kappa , \in, A \rangle \prec \langle V_\gamma, \in, A^* \rangle$ is a proper elementary extension.\footnote{As described by Hamkins and Johnstone in the comments on page 5 of their paper, we may let $\gamma$ be regular, uplifting, weakly compact, etc.} \end{definition}

If $\kappa$ is strongly uplifting then $\kappa$ is inaccessible.

%
%
%

\begin{theorem} [{\cite[Thm.~19]{Hamkins:2014yu}}]
Let $\kappa$ be strongly uplifting. Then there are forcing extensions in which the following hold:
\begin{enumerate}
	\item $\bfRA$ and $\kappa=\c=\aleph_1$.	
	\item $\bfRAccc$ and $\kappa=\c$.
	\item $\bfRAp$ and $\kappa = \c = \aleph_2$.
	\item $\bfRAc$ and $\kappa=\aleph_2$ and $\CH$.
	\item {\cite[Thm.~4.2.12]{Minden:2017fr}} $\bfRAsc$ and $\kappa=\aleph_2$ and $\CH$.
\end{enumerate}
\end{theorem}

\begin{theorem}[{\cite[Thm.~19]{Hamkins:2014yu}}] \label{thm:stronglyuplifting=RA}
We have the following implications:
	\begin{enumerate}
		\item $\bfRA$ $\implies$ $\aleph_1^V$ is strongly uplifting in $L$.
		\item $\bfRAccc$ $\implies$ $\c^V$ is strongly uplifting in $L$.
		\item $\bfRAp$ $\implies$ $\c^V=\aleph_2^V$ is strongly uplifting in $L$.
		\item $\bfRAc$ $\implies$ $\aleph_2^V$ is strongly uplifting in $L$.
		\item {\cite[Thm~4.2.13]{Minden:2017fr}} $\bfRAsc$ $\implies$ $\aleph_2^V$ is strongly uplifting in $L$.
	\end{enumerate}
\end{theorem}

\section{Combining Resurrection and Maximality}
\label{sec:RA+MP}
Hamkins and Johnstone \cite[Section 6]{Hamkins:2013qv} combine the resurrection axiom with forcing axioms, like $\PFA$ for example, and show that they both hold after a forcing iteration. Fuchs explores combinations of maximality principles for closed forcing notions \cite{Fuchs:2008rt} combined with those for collapse and directed closed forcing notions \cite{Fuchs:2008ve} and hierarchies of resurrection axioms with an emphasis on subcomplete forcing \cite{Fuchs:2018RA}. Trang and Ikegami  \cite{Trang:2017} studied classes of maximality principles combined with forcing axioms.
This section focuses on a different question in a similar vein: is it possible for the resurrection axiom and the maximality principle to hold at the same time? The axioms do not imply each other directly,
$\bfMP$ surely does not imply $\bfRA$, since the consistency strength of $\bfMP$ is that of a fully reflecting cardinal, while $\bfRA$ has the consistency strength a strongly uplifting cardinal. These two cardinals are consistently different, they certainly don't imply each other: if $\kappa$ is fully reflecting, take the least $\gamma$ such that $V_\kappa \prec V_\gamma$. If there isn't such a $\gamma$, then $\kappa$ isn't uplifting anyway. But in $V_\gamma$, we have that $\kappa$ is not uplifting.
There is no implication in the other direction as well. This is because working in a minimal model of 
	$$ \ZFC \; + \; ``V=L" \; + \; \text{ ``there is an uplifting cardinal"}$$ 
(i.e., no initial segment of the model satisfies this theory), we may force over this minimal model to obtain $\RA$. Now $\bfMP$ can't hold in the extension, since letting $\kappa$ be the $\omega_2$ of the extension, if $\bfMP$ were true, then that would imply that $L_\kappa$ is elementary in $L$ -- contradicting the minimality of the model we started with.

This section is dedicated to showing that it is possible for maximality and resurrection to both hold, by combining the techniques showing the consistency of each principle, all in one minimal counterexample iteration. 

An inaccessible cardinal $\kappa$ is \emph{strongly uplifting fully reflecting} so long as it is both strongly uplifting and fully reflecting.

Combining these two large cardinal notions is almost natural. If $\kappa$ is uplifting then there are unboundedly many $\gamma$ such that $V_\kappa \prec V_\gamma$, and if on top of that $\kappa$ is fully reflecting, we add that $V_\kappa \prec V$ as well, where $V$ is in some sense the limit of the $V_\gamma$'s. Moreover, strongly uplifting fully reflecting cardinals are guaranteed to exist in some set-sized transitive model if there are subtle cardinals.

\begin{definition} A cardinal $\delta$ is \emph{subtle} so long as for any club $C \subseteq \delta$ and for any sequence $\mathcal A = \seq{ A_\alpha }{ \alpha \in C }$ with $A_\alpha \subseteq \alpha$, there is a pair of ordinals $\alpha < \beta$ in $C$ such that $A_\alpha=A_\beta \cap \alpha$. \end{definition}

\begin{fact} \label{fact:subtleinacc}
If a cardinal $\delta$ is subtle, then $\delta$ is inaccessible.
\end{fact}
%

\begin{proposition} If $\delta$ is subtle, then it is consistent that there is a strongly uplifting fully reflecting cardinal. Namely, the set $$\set{ \kappa < \delta }{ \text{$V_\delta \models ``\kappa$ is strongly uplifting and $V_\kappa \prec V$"} }$$ is stationary in $\delta$. \end{proposition} 
\begin{proof}
Hamkins and Johnstone \cite[Thm.~7]{Hamkins:2014yu} show that if $\delta$ is subtle, then the set of cardinals $\kappa$ below $\delta$ that are strongly uplifting in $V_\delta$ is stationary. But since $\delta$ is subtle, it must also be inaccessible by Fact \ref{fact:subtleinacc}. Thus in $V_\delta$, by the proof of the downward Lowenheim-Skolem theorem, there is a club $C \subseteq \delta$ of cardinals $\kappa$ such that $V_\kappa \prec V_\delta$; meaning that $\kappa$ is fully reflecting in $V_\delta$. This means that there is some $\alpha < \kappa$ that is both strongly uplifting and fully reflecting in $V_\delta$, giving us the required consistency.
\end{proof}

We can immediately see that if resurrection and maximality both hold, we must have a strongly uplifting fully reflecting cardinal, by combining the results of \ref{thm:fullyreflecting=MP} and \ref{thm:stronglyuplifting=RA}.
\begin{observation}
The following consistency results hold.
\begin{enumerate}
	\item $\bfMP + \bfRA \implies \aleph_1^V$ is strongly uplifting fully reflecting in $L$.
	\item $\bfMPccc + \bfRAccc \implies \c^V$ is strongly uplifting fully reflecting in $L$.
	\item $\bfMPp + \bfRAp \implies \aleph_2^V$ is strongly uplifting fully reflecting in $L$.
	\item $\bfMPc + \bfRAc \implies \aleph_2^V$ is strongly uplifting fully reflecting in $L$.
	\item $\bfMPsc + \bfRAsc \implies \aleph_2^V$ is strongly uplifting fully reflecting in $L$.
\end{enumerate}
\end{observation}

For the other direction of the consistency result, we restate here a restricted version of the lifting lemma \cite[Lemma 17]{Hamkins:2014yu}, using the comments following it, to allow the lifting of certain embeddings to generic extensions. 

\begin{fact}[Lifting Lemma] \label{fact:liftinglemma}
Suppose that $\langle M, \in A \rangle  \prec \langle M^*, \in, A^* \rangle$ are transitive models of $\ZFC$ and $\P$ is an $Ord^M$-length forcing iteration without any condition with full support in $M$. If $G \subseteq \P$ is an $M$-generic filter and $G^* \subseteq \P^*$ is $M^*$-generic with $G = G^* \cap \P$, then $\langle M[G], \in, A, G \rangle \prec \langle M^*[G^*], \in, A^*, G^* \rangle.$
\end{fact}

\begin{theorem}\label{main} Let $\kappa$ be a strongly uplifting fully reflecting cardinal. Then there is a forcing extension in which both $\bfRA$ and $\bfMP$ hold, and $\kappa=\c=\aleph_1$. \end{theorem}
\begin{proof}
Let $\kappa$ be strongly uplifting fully reflecting. Below we define $\P$ to be the least-counterexample to $\bfRA + \bfMP$ lottery sum finite support 
iteration of length $\kappa$. We generically pick, using the lottery sum, whether at each stage to force with a least-rank counterexample to the maximality principle or a least-rank counterexample to the boldface resurrection axiom.

In particular, define the poset $\P = \P_\kappa = \seq{ ( \P_\alpha, \dot{\Q}_\alpha ) }{ \alpha < \kappa } $ as follows: 

At stage $\alpha$, consider all of the sentences with parameters having names in $V_\kappa^{\P_\alpha}$ that are not true in $V_\kappa^{\P_\alpha}$, but can be forced by some poset $\dot \Q$ to be necessary. Let $\mathcal M$ be the collection of such possible forcing notions $\dot \Q$ in $V_\kappa^{\P_\alpha}$ of minimal rank in $V^{\P_\alpha}_\kappa$. In other words, $\mathcal M$ contains the current minimal rank counterexamples to the boldface maximality principle, as seen by $V_\kappa$. Since $\kappa$ is fully reflecting, $V_\kappa$ and $V$ are in agreement on which posets are counterexamples to maximality, $\mathcal M$ can be thought of as the collection of current minimal rank counterexamples to $\bfMP$.

Similarly we take $\mathcal R$ to contain all of the current minimal rank counterexamples to $\bfRA$. To be precise, let $\mathcal R$ be the collection of forcing notions $\dot \Q$ of minimal rank such that there is $\dot A \subseteq H^{V^{\P_\alpha}}_{\c}$ where after any further forcing $\ddot \R$, it is not the case that $\langle H^{V^{\P_\alpha}}_{\c}, \in, \dot A \rangle  \prec \langle H_{\c}^{V^{\P_\alpha *\dot \Q *\ddot \R}}, \in, \dot A^* \rangle$ for any $\dot A^*$. 

Then take
$\P_{\alpha+1} = \P_\alpha * \dot{\Q}_\alpha$ where $\dot{\Q}_\alpha$ is a term for the lottery sum $\bigoplus \mathcal R \oplus \bigoplus \mathcal M$.
Limit stages are taken care of via finite support.

Let $G \subseteq \P$ be generic. 

Since it is dense in this iteration for sets of size less than $\kappa$ to be collapsed to be countable (indeed, for any $a \in V_\kappa$, the statement positing $a$ is countable is forceably necessary) it follows that $\kappa \leq \omega_1$ in $V[G]$. It is straightforward to show that the iteration has the $\kappa$-$cc$. Thus $\kappa$ remains a cardinal, so $\kappa = \aleph_1$ in the extension. Additionally, a density argument shows that unboundedly often in the iteration, a new real is added. (Alternatively, the finite support will itself add a new real at every limit stage of cofinality $\omega$.) Either way, this means that $\kappa=\aleph_1=\c$ in $V[G]$ as claimed.

We need to show that both $\bfRA$ and $\bfMP$ hold in $V[G]$. 

\begin{claimno} $\bfMP$ holds in $V[G]$. \end{claimno}
\begin{proof}[Pf]
Assume it fails; the sentence $\varphi(\vec a)$, where $\vec a \in H_{\omega_1}^{V[G]}$ is a parameter set, has the property that:
	$\text{$V[G] \models ``\varphi(\vec a)$ is forceably necessary but $\varphi(\vec a)$ is false."}$
Choose a condition $p \in G$ that forces the above statement.
$\P$ has the $\kappa$-$cc$, since all of the forcing notions are small -- so at no stage in the iteration is $\kappa$ collapsed. This means that there has to be some stage where $\vec a$ appears.
So there is some stage in the iteration beyond the support of $p$, say $\alpha < \kappa$, where $\vec a \in V_\kappa[G_\alpha]$. Specifically $\varphi(\vec a)$ is an available button at stage $\alpha$, since after the rest of the iteration, $\Ptail$ where $\P=\P_\alpha * \Ptail$, we have that $\varphi(\vec a)$ is forceably necessary. Indeed, this is reasoning available in $V[G_\alpha]$, which thus sees that $\varphi(\vec a)$ is a button. By elementarity, as $\kappa$ is fully reflecting, it follows that
	$V_\kappa[G_\alpha] \models ``\varphi(\vec a) \text{ is forceably necessary"}.$
From that point on, $\varphi(\vec a)$ continues to be a button, since we have that $\alpha$ is beyond the support of $p \in G$.
Thus it is dense, in $\P$, for $\varphi(\vec a)$ to be ``pushed" at some point after stage $\alpha$, say $\beta$. 
So we have $\beta < \kappa$ such that there is some $\Q$ forcing $\varphi(\vec a)$ to be necessary in $V_\kappa[G_\beta]$. Let $H \subseteq \Q$ be generic over $V[G_\beta]$ so that there is some $\Gtail$ generic for the rest of $\P$ satisfying $V[G_\beta][H][\Gtail]=V[G]$. The sentence $\varphi(\vec a)$ is now necessary in $V_\kappa[G_\beta][H]$. But then since $V_\kappa[G_\beta][H] \prec V[G_\beta][H]$, as we are still in an initial segment of the full iteration, we have that $\varphi(\vec a)$ is necessary in $V[G_\beta][H]$, by elementarity. Therefore since the rest of the iteration is a forcing notion in its own right, $\varphi(\vec a)$ is true in $V[G_\beta][H][\Gtail]=V[G]$, contradicting our assumption that $\varphi(\vec a)$ is false in $V[G]$.
\end{proof}

\begin{claimno} 
$\bfRA$ holds in $V[G]$
\end{claimno}
\begin{proof}[Pf]
Assume toward a contradiction that $\bfRA$ fails. This means we can choose a least rank counterexample, a forcing $\Q$ in $V[G]$ which supposedly cannot be resurrected. Let $A \subseteq H^{V[G]}_{\c}$ be its associated predicate. Let $\dot{\Q}$ be a name for $\Q$ of minimal rank. Since $\P$ has the $\kappa$-$cc$, there must be a name for the predicate in the extension such that $\dot A \subseteq H_{\kappa}$ with $A = \dot A^G$.

We will argue that $\dot \Q$ appears at stage $\kappa$ of the same exact iteration, except defined in some larger $V_\gamma[G]=V_\gamma^{V[G]}$ where $\gamma$ is inaccessible. 
Use the strong uplifting property of $\kappa$, and code the iteration $\P$ as a subset of $\kappa$, to find a sufficiently large inaccessible cardinal $\gamma$ so that 
	$\langle V_\kappa, \in, \P, \dot A \rangle \prec \langle V_\gamma, \in, \P^*, \dot A^* \rangle,$
where $\P^*$ is the least-counterexample to $\bfRA+\bfMP$ lottery sum iteration of length $\gamma$ as defined in $V_\gamma$. Obtaining a large enough $\gamma$ involves a process of closing under least-rank counterexamples. Not only does $V_\gamma$ need to agree with $V$ about the rank of least-rank counterexamples to the resurrection axiom throughout the iteration $\P$, it must also compute the least-rank counterexamples to the maximality principle appropriately as well. Since $V_\kappa \in V_\gamma$, and the ranks throughout the maximality iteration were computed in $V_\kappa$, the minimal ranks are guaranteed to be computed properly in $V_\gamma$.
Indeed we have argued above that $\P^*$ is defined the same way as $\P$ below stage $\kappa$, so we may assume below a condition in $G$ that $\dot \Q$ may be picked at stage $\kappa$. So below a condition that opts for $\dot \Q$ at the stage $\kappa$ lottery we may say that $\P^*$ factors as $\P * \dot{\Q} * \dot{\Ptail^*}$. 
Let $H*\Gtail^* \subseteq \Q * \Ptail^*$ be $V[G]$-generic. Letting $G^*=G*H*\Gtail^*$, this means that $G^* \subseteq \P^*$ is  generic over $V$. 

Thus by the lifting lemma (Fact \ref{fact:liftinglemma}) the strongly uplifting embedding 
	$\langle V_\kappa, \in, \P, \dot A \rangle \prec \langle V_\gamma, \in, \P^*, \dot A^* \rangle$
lifts to 
	$\langle V_\kappa[G], \in, \P, \dot A, G \rangle \prec \langle V_\gamma[G^*], \in, \P^*, \dot A^*, G^* \rangle $
in $V[G^*]$. In particular,
	$\langle V_\kappa[G], \in, \P, A \rangle \prec \langle V_\gamma[G^*], \in, \P^*, A^* \rangle.$
We have that 
	$V_\kappa[G]=H_\kappa^{V[G]} = H_\c^{V[G]},$
since $\kappa$ is inaccessible and $\P$ has the $\kappa$-$cc$. We can argue the same way as above, replacing $\kappa$ with $\gamma$, to get that 
	$V_\gamma[G^*]=H_\gamma^{V[G^*]} = H_\c^{V[G^*]}.$ 
This establishes 
	$\langle H_\c^{V[G]}, \in, A \rangle \prec \langle H_\c^{V[G^*]}, \in, A^* \rangle.$ 
\end{proof}
Therefore $\bfRA$ and $\bfMP$ both hold as desired.
\end{proof}

\begin{theorem} \label{thm:RA+MP} Let $\kappa$ be a strongly uplifting fully reflecting cardinal. Then there are forcing extensions in which we have the following:
\begin{enumerate}
	\item\label{item:proper} $\bfRAp + \; \bfMPp + \; \kappa= \c = \aleph_2$. 
	\item\label{item:ctblyclosed} $\bfRAc + \; \bfMPc + \; \kappa=\aleph_2 + \; \CH$. 
	\item\label{item:subcomplete} $\bfRAsc + \; \bfMPsc + \; \kappa=\aleph_2+ \; \CH$. 
\end{enumerate}	
\end{theorem}

\begin{proof}
For all of these arguments, use Theorem \ref{main} as a blueprint. The definition of the iteration is always the same, with the caveat that the forcing notions are always taken to be in the relevant class we are thinking about (where $\Gamma$ is the class of proper forcing notions when we show (\ref{item:proper}), and so on). Let's repeat the description of the iteration as before, relativized to a forcing class $\Gamma$.

Each time we define an iteration $\P = \P_\kappa = \seq{ ( \P_\alpha, \dot{\Q}_\alpha ) }{ \alpha < \kappa } $ as follows: 

At stage $\alpha$, consider all of the sentences with parameters having names in $V_\kappa^{\P_\alpha}$ that are not true in $V_\kappa^{\P_\alpha}$, but can be forced by some poset $\dot \Q \in \Gamma^{V_\kappa^{\P_\alpha}}$ to be necessary. Let $\mathcal M$ be the collection of such possible forcing notions of minimal rank in $V^{\P_\alpha}_\kappa$ for which the above holds. So $\mathcal M$ contains the minimal rank counterexamples to $\bfMP_\Gamma$.

Additionally, let $\mathcal R$ be the collection of forcing notions $\dot \Q \in \Gamma^{V^{\P_\alpha}}$ of minimal rank such that there is $\dot A \subseteq H^{V^{\P_\alpha}}_{\c}$ where after any further forcing $\ddot \R \in \Gamma^{V^{\P_\alpha*\dot{\Q}}}$ and for all $\dot A^*$, it is not the case that $\langle H^{V^{\P_\alpha}}_{\c}, \in, \dot A \rangle  \prec \langle H_{\c}^{V^{\P_\alpha *\dot \Q *\ddot \R}}, \in, \dot A^* \rangle$. So $\mathcal R$ contains the minimal rank counterexamples to $\bfRA_\Gamma$.

Then take
$\P_{\alpha+1} = \P_\alpha * \dot{\Q}_\alpha$ where $\dot{\Q}_\alpha$ is a term for the lottery sum $\bigoplus \mathcal R \oplus \bigoplus \mathcal M$.

The argument for each forcing class follows the above template, granted a useful support is used. For (\ref{item:proper}) and (\ref{item:ctblyclosed}), use countable support. 
For (\ref{item:subcomplete}), revised countable support works, after seeing that such iterations of subcomplete forcing are subcomplete \cite{Jensen:2014}.

All of the forcing classes considered here contain the forcing to collapse the size of sets to $\omega_1$, and in each iteration, it is dense to do so. This means that $\kappa \leq \omega_2$ in each of these extensions. It also is not difficult to show that each of these iterations have the $\kappa$-$cc$. Thus $\kappa=\aleph_2$ in each of these forcing extensions. With proper forcing, a density argument similar to that of the proof of Theorem \ref{main} tells us $\kappa=\c$ in the extension. For countably closed and subcomplete forcing, in fact $\bfMPc$, $\bfRAc$, $\bfRAsc$, and $\bfMPsc$ all imply $\CH$, so $\CH$ will hold in these extensions.
\end{proof}

The iteration as defined is clearly ill-suited for $ccc$ forcings, since $ccc$ forcings are not closed under lottery sums. 
To resolve this let's  combine the arguments from \cite{Hamkins:2014yu} and \cite{Hamkins:2003jk}. This method could also be used for other forcing classes. 



\begin{definition} \cite{Hamkins:2014yu}
 A \emph{Laver function $\ell$ for a strongly uplifting cardinal $\kappa$} is a partial function from $\kappa$ to $V_\kappa$ satisfying that for every $A \subseteq \kappa$, every ordinal $\theta$, and every set $x$, there is a proper elementary extension $\langle V_\kappa, \in, A, \ell \rangle \prec \langle V_\gamma, \in, A^*, \ell^* \rangle$ where $\gamma \geq \theta$ is inaccessible and $\ell^*(\kappa)=x$.
 \end{definition}
 
 By \cite[Thms.~2~\&~11]{Hamkins:2014yu} every strongly uplifting cardinal has such a Laver function in $L$.

\begin{theorem} \label{thm:cccRA+MP} Let $\kappa$ be a strongly uplifting fully reflecting cardinal. Then there is a forcing extension of $L$ in which both $\bfRAccc$ and $\bfMPccc$ hold, and $\kappa = \c$.
\end{theorem}
\begin{proof}
Let $\kappa$ be strongly uplifting fully reflecting. By \cite[Thm.~2]{Hamkins:2014yu}, $\kappa$ is strongly uplifting in $L$. Since $L \models V_\kappa = L_\kappa$, it follows that $\kappa$ is fully reflecting in $L$ as well. Work in $L$. 

Let $\ell$ be a strongly uplifting Laver function for $\kappa$. Let $\vec{\varphi}=\seq{ \varphi_\alpha(\dot{\vec a}) }{\alpha <\kappa}$ enumerate, with unbounded repetition, all sentences in the language of set theory with names for parameters in $V_\kappa$ coming from a forcing extension by a forcing of size less than $\kappa$. 

Define a finite support $\kappa$-iteration of $ccc$ forcing so that at successor stages $\alpha=\beta+1$, the forcing $\Q_\alpha$ is least rank forcing $\varphi_\beta$ to be necessary over $V_\kappa^{\P_\alpha}$, if possible. Otherwise, do trivial forcing at that stage. At limit stages, force with $\ell(\alpha)$, provided that this is a $\P_\alpha$-name for a $ccc$ forcing. Otherwise, do trivial forcing at that stage.

Unboundedly often, the forcing will not be trivial in this iteration. Forcing to make $\c$ arbitrarily large below $\kappa$ will happen periodically at successor stages. At limits, the Laver function will choose Cohen forcing unboundedly often. 

Let $G \subseteq \P$ be generic. We have that $\kappa=\c^{V[G]}$, which follows from the above (and also since $\P$ is a finite support iteration of $ccc$ forcings, a Cohen real will be added in the forcing up to limit stages of countable cofinality). Since there are $\kappa$-many such stages, $\kappa$-many Cohen reals are added. 

\begin{claimno} $\bfMPccc$ holds in $V[G]$. \end{claimno}
\begin{proof}[Pf.]
Let $\varphi(\vec a)$ be $ccc$-forceably $ccc$-necessary over $V[G]$ with $\vec a \in H_\c^{V[G]} = H_\kappa^{V[G]}$. Since $\P$ is $ccc$, $\vec a$ has a $\P_\alpha$-name $\dot{\vec a}$, where $\beta+1=\alpha<\kappa$ is some successor ordinal and $\varphi(\dot{\vec a})=\varphi_\beta(\dot{\vec a})$. Since $\varphi(\vec a)$ is $ccc$-forceably $ccc$-necessary in $V[G]$ it must have also been $ccc$-forceably $ccc$-necessary in such a $V[G_\alpha]$ (by adding the rest of the iteration to the beginning of whatever forcing notion makes the sentence $ccc$-necessary in $V[G]$). Since $V_\kappa[G_\alpha]\prec V[G_\alpha]$, the sentence must have been $ccc$-forceably $ccc$-necessary in $V_\kappa[G_\alpha]$. Thus $\varphi(\vec a)$ was forced to be $ccc$-necessary; and so it is $ccc$-necessary in $V[G_{\alpha+1}]$. Since the rest of the iteration is $ccc$, it is true in $V[G]$.
\end{proof}

\begin{claimno} $\bfRAccc$ holds in $V[G]$.\end{claimno}
\begin{proof}[Pf.]
Suppose that $A\subseteq \kappa$ and $\Q$ is a $ccc$ forcing in $V[G]$. Let $\dot A$ and $\dot \Q$ be $\P$-names for $A$ and $\Q$ respectively. Let $\theta$ be an ordinal. Since $\ell$ is a strongly uplifting Laver function for $\kappa$, there is an extension $\langle V_\kappa, \in, \dot A, \P, \vec{\varphi}, \ell \rangle \prec \langle V_\gamma, \in, \dot A^*, \P^*, \vec{\varphi}^*, \ell^*\rangle$ with $\gamma \geq \theta$ inaccessible and $\ell^*(\kappa)=\dot \Q$. Note that $\P$ is definable from $\vec{\varphi}$ and $\ell$. Thus $\P^*$ is the corresponding $\gamma$-iteration defined from $\ell^*$ and $\vec{\varphi}^*$. Furthermore $\P^*=\P*\dot \Q *\dot{\Ptail^*}$, where $\dot{\Ptail^*}$ is the rest of the iteration after stage $\kappa$ up to $\gamma$, which is $ccc$ in $V[G][H]$ ($H\subseteq \Q$ generic over $V[G]$) since it is a finite support iteration of $ccc$ forcing. Let $G^*\subseteq \P^*$ be generic over $V$, containing $G*H$. By Fact \ref{fact:liftinglemma}, we may lift the elementary extension to $\langle V_\kappa[G], \in, \dot A, \P, \vec{\varphi}, \ell, G \rangle \prec \langle V_\gamma[G^*], \in, \dot A^*, \P^*, \vec{\varphi}^*, \ell^*, G^*\rangle$. As $\gamma = \c^{V[G^*]}$, the desired result follows.
\end{proof}
Therefore $\bfRAccc$ and $\bfMPccc$ hold in an extension of $L$ as desired.
\end{proof}

\section{Local Maximality} 
\label{sec:LMP}
The local maximality principle is a natural axiom which elucidates somewhat how the resurrection axiom and the maximality principle intersect. It is one kind of intermediate step between the maximality principle and bounded forcing axiom.


In the local version of the maximality principle, the truth of a forceably necessary sentence will be checked not in $V$ but in a much smaller structure. This should be compared to one of the equivalent ways of defining the bounded forcing axiom, namely generic absoluteness. 

For the appropriate definition of the bounded forcing axiom $\BFA{\kappa}{\Gamma}$ refer to \cite[Definition 2]{Bagaria:2000fj}.

\begin{definition} \label{def:GenericAbsoluteness} Let $n$ be a natural number, $\Gamma$ be a class of forcing notions, and $M$ be a transitive set (usually either $\omega_1$ or $H_{\omega_1}$). Then \emph{$\Gamma$-generic $\Sigma_n(M)$-absoluteness} with parameters in $S \subseteq \mathcal P(M)$ is the statement that for any $\Sigma_n$-sentence $\varphi(\vec a)$ where $\vec a \in S \cap M$ and predicate symbols $\vec{\dot A}$, the following holds: Whenever $\vec A \in S \cap \mathcal P(M)$, $\P \in \Gamma$, and $G$ is $\P$-generic over $V$, then $$(\langle M, \in, \vec A \rangle \models \varphi(\vec a))^V \iff (\langle M, \in, \vec A \rangle \models \varphi(\vec a) )^{V[G]},$$
where $\vec{\dot A}$ is meant to be interpreted in $M$ as $\vec A$, and the satisfaction is first order.
\end{definition}

\begin{fact}[\cite{Bagaria:2000fj}] \label{Fact:BFA} Let $\kappa$ be an infinite cardinal of uncountable cofinality and let $\Gamma$ be a class of forcing notions. Then the following are equivalent:
\begin{enumerate} 
	\item $\BFA{\kappa}{\Gamma}$ 
	\item For every $\P \in \Gamma$ and generic $G \subseteq \P$,  $H_{\kappa^+} \prec_{\Sigma_1} H_{\kappa^+}^{V[G]}.$ 
	\item \label{item:GammaGenericSigma1Absoluteness} $\Gamma$-generic $\Sigma_1(H_{\kappa^+})$-absoluteness.\end{enumerate}	
	\end{fact}
	
The bounded forcing axiom $\BFA{\kappa}{\Gamma}$ where $\kappa =2^\omega$ and $\Gamma$ is the class of $ccc$ forcing notions is equivalent to $\MA_\kappa$. We look at $\kappa=\omega_1$. The bounded forcing axiom for the class of countably closed forcing is a theorem of $\ZFC$. 
We write $\BSCFA$ for when $\Gamma=\set{ \P }{ \P \text{ is subcomplete} }$, and $\BPFA$ for proper forcing.
The consistency strength of each of these bounded forcing axioms is exactly that of a reflecting cardinal (as is shown by Goldstern and Shelah \cite{Goldstern:1995zr} for proper forcing and the subcomplete version is shown by Fuchs \cite{Fuchs:2016yo}). 

In the following, we write $\varphi^M(\vec a)$ for the sentence $M \models \varphi(\vec a)$.

\begin{definition} Let $\Gamma$ be a  class of forcing notions, and let $S$ be a set of parameters. Let $M$ be a defined term for a structure to be reinterpreted in forcing extensions, and $S \subseteq M$. The \emph{Local Maximality Principle} relative to $M$ ($\MP^{M}_{\Gamma} (S)$) is the statement that for every parameter set $\vec a \in S$ and every sentence $\varphi(\vec a)$, if $\varphi^{M}(\vec a)$ is $\Gamma$-forceably $\Gamma$-necessary, then $\varphi^{M}(\vec a)$ is true. 
\end{definition}

Write $\bflMP$ for the local version of $\bfMP$, and $\bflMP_\Gamma$ for the local version of $\bfMP_\Gamma$ with forcing classes $\Gamma$. As before and as discussed in Lemmas \ref{lemma:MPscparams} and \ref{lemma:MPcccparams}, the choice of $H_{\omega_2}$ makes sense for the parameter set for the boldface subcomplete, countably closed, and proper maximality principles, and the choice of $H_\c$ makes sense for the boldface $ccc$ maximality principle. Additionally, the smallest model $M$ that makes sense to use for the local version has to at least contain the parameter set, so $S=M$ is what we will work with here.

Clearly for forcing classes $\Gamma$, $\MP_{\Gamma}(H_\kappa) \implies \lMP{\Gamma}{\kappa}$.

\begin{proposition} \label{prop:RA->LMP} For $\Gamma$ a class of forcing notions, $\RA_\Gamma(H_{\kappa}) \implies \lMP{\Gamma}{\kappa}$. \end{proposition}
\begin{proof}
Suppose that $\RA_\Gamma(H_{\kappa})$ holds. To see that the local maximality principle holds, suppose that $\varphi(\vec a)$ is a sentence such that the sentence ``$H_{\kappa} \models \varphi(\vec a)$" is $\Gamma$-forceably $\Gamma$-necessary. So there is a forcing notion $\P \in \Gamma$ such that after any further forcing, we have that ``$H_{\kappa} \models \varphi(\vec a)$" holds in the two-step extension. By resurrection, there is a further $\dot \R$ such that $\forces_\P ``\text{$\dot \R \in \Gamma$}"$ and letting $G*h \subseteq \P *\dot \R$ be generic we have $H_{\kappa} \prec H_{\kappa}^{V[G*h]}$. Since ``$H_{\kappa} \models \varphi(\vec a)$" in the two-step extension $V[G][h]$ by our assumption, this means that $H_{\kappa} \models \varphi(\vec a)$ holds by elementarity, so $\lMP{\Gamma}{\kappa}$ holds as desired.
\end{proof}

\begin{proposition} For $\Gamma$ a class of forcing notions, $\lMP{\Gamma}{\kappa^+} \implies \BFA{\kappa}{\Gamma}$. \end{proposition}
\begin{proof}
Assume that $\lMP{\Gamma}{\kappa^+}$ holds. We use characterization \ref{item:GammaGenericSigma1Absoluteness} of $\BFA{\kappa}{\Gamma}$ from Fact \ref{Fact:BFA}. To show that $\Gamma$-generic $\Sigma_1(H_{\kappa^+})$-absoluteness holds, let $\varphi(\vec x)$ be a $\Sigma_1$-formula, $\vec a \in H_{\kappa^+}$, and $\P \in \Gamma$, satisfying $\forces_\P \varphi(\check{\vec{a}})$. Let $G \subseteq \P$ be generic. 
Since $\varphi(\vec x)$ is $\Sigma_1$ and $H_{\kappa^+} \prec_{\Sigma_1} V$ as $\kappa^+$ is regular, we have that $\varphi^{H_{\kappa^+}}(\vec a)= \varphi^{H_{\kappa^+}}((\check{\vec{a}})^G)$ holds in all future forcing extensions.
Thus $\varphi^{H_{\kappa^+}}(\vec a)$ is $\Gamma$-forceably $\Gamma$-necessary, which means that $\varphi^{H_{\kappa^+}}(\vec a)$ is true (in $V$) by the local maximality principle. Thus $\varphi(\vec a)$ holds in $V$ as desired.
\end{proof}

\begin{proposition} \label{proposition:localproperties} $\bflMPsc$ (and $\bflMPc$) implies: \begin{enumerate}
	\item \label{item:Suslin} There is a Suslin tree.
	\item \label{item:diamond} $\lozenge$ holds.
	\item \label{item:CH} $\CH$ holds.
\end{enumerate}\end{proposition} 
\begin{proof}
Firstly, $H_{\omega_2}$ is enough to verify each of these properties.

For \ref{item:Suslin}, note that the forcing to add a Suslin tree is countably closed and thus is subcomplete as well. But any particular Suslin tree will continue to be a Suslin tree after any subcomplete (or countably closed) forcing \cite[Ch.~3 p.~10]{Jensen:2009wc}. Thus the existence of a Suslin tree is $sc$-forceably $sc$-necessary, and hence true by $\bflMPsc$ (and likewise for countably closed forcing). 

For \ref{item:diamond} note that Jensen \cite[Ch.~3 p.~7]{Jensen:2009wc} shows that $\lozenge$ will hold after performing subcomplete forcing if it held in the ground model. Since forcing to add a $\lozenge$-sequence is countably closed, and $\lozenge$ will continue to hold after any subcomplete (or countably closed) forcing, so must be true by the relevant boldface maximality. Of course then \ref{item:CH} follows, since $\lozenge$ implies $\CH$.
\end{proof}

\begin{proposition} $\bflMPp$ (and $\bflMPccc$) implies: \begin{enumerate}
	\item There are no Suslin trees.
	\item All Aronszajn trees are special.
\end{enumerate}
\end{proposition} 
\begin{proof}
Forcing with a Suslin tree is proper (indeed, $ccc$), and adds a branch through the tree making it fail to be Suslin. Additionally specializing Aronszajn trees is proper, in fact $ccc$, and kills Suslin trees, but also once a specializing function is added it can't be removed by further proper (or $ccc$) forcing.
\end{proof}
%

\subsection{Consistency of the Local Maximality Principle}
\label{subsec:ConlocalMP}
We will now introduce the large cardinal property that is equiconsistent with the local maximality principle. When showing the consistency of the resurrection axiom in section \ref{sec:RA}, we defined the notion of an \textit{uplifting cardinal}, of which the following property is the suitable ``local" version.

\begin{definition} An inaccessible cardinal $\delta$ is \emph{locally uplifting} so long as for every $\theta$ we have that $\delta$ is $\theta$-locally uplifting, meaning that for every formula $\varphi(x)$ and $a \in V_\delta$, there is an inaccessible $\gamma > \theta$ such that 
	$V_\delta \models \varphi(a) \iff V_\gamma \models \varphi(a).$ \end{definition}

Clearly if $\kappa$ uplifting or fully reflecting, then $\kappa$ is locally uplifting. 

Note that if a regular cardinal $\delta$ has the property of being locally uplifting, without necessarily being inaccessible, then $\delta$ must be inaccessible, since otherwise if $2^\alpha \geq \delta$ for some $\alpha<\delta$, this is seen by some larger $V_\gamma$, i.e., $V_\gamma \models \exists \beta \ \left[ 2^\alpha = \beta \right]$. So by elementarity there is some $\beta' = 2^\alpha$ in $V_\delta$, a contradiction.

It is not hard to see that if $\delta$ is locally uplifting, then it is locally uplifting in $L$: for any formula $\varphi(x)$ and $a \in L_\delta$, we would have that $V_\delta \models \varphi^L(a) \iff V_\gamma \models \varphi^L(a)$, which implies that $L_\delta \models \varphi(a) \iff L_\gamma \models \varphi(a)$.

We have the following relationship between locally uplifting and reflecting cardinals.

\begin{definition}
	We say that a cardinal $\kappa$ is \emph{reflecting} so long as for any regular cardinal $\theta$ and for any formula $\varphi(x)$, if $\vec a \in H_\kappa$ and $H_\theta \models \varphi(\vec a)$, then there is $\delta<\kappa$ satisfying $H_\delta \models \varphi(\vec a)$.
\end{definition}

\begin{proposition} If $\kappa$ is locally uplifting then $\kappa$ is reflecting. \end{proposition}
\begin{proof} 
Suppose that $\kappa$ is locally uplifting. To show that $\kappa$ is reflecting, let $\varphi(x)$ be a formula and let $\vec a \in H_\kappa$, and assume that there is $\theta > \kappa$ where $H_\theta \models \varphi(\vec a)$. Define $\psi$ as follows:
	$$\psi(\vec a): \ \  \exists \delta \ \left[ H_\delta \models \varphi(\vec a) \right].$$
Then if we take $\gamma>\theta$ satisfying $H_\theta \in H_\gamma$, we have that $H_\gamma \models \psi(\vec a)$. As $\kappa$ is locally uplifting, this implies that $H_\kappa \models \psi(\vec a)$. Thus there is $\delta < \kappa$ such that $H_\delta \models \varphi(\vec a)$ as desired.
\end{proof}

The local maximality principle is equiconsistent with the existence of a locally uplifting cardinal, using the same method as with the proof of the maximality principle but with some care in relativizing to $H_{\omega_2}$.

\begin{theorem}\label{theorem:bflMPinL} If $\bflMP$ holds, then $\aleph_1^V$ is locally uplifting in $L$. \end{theorem}	
\begin{proof}
Let $\kappa=\aleph_1^V$ and suppose that the local maximality principle holds.

Firstly, $\kappa$ is a limit cardinal in $L$, since for $\gamma < \kappa$, the statement $H_{\omega_1} \models$ ``there is a cardinal in $L$ greater than $\gamma$" is forceably necessary (by taking $\Coll(\omega, \kappa)$) and thus true in $H_{\omega_1}$. So we have that $\kappa$ is inaccessible in $L$.

Assume $L_\kappa \models \varphi(\vec a)$. In other words, $H_{\omega_1}^V \models \varphi^L(\vec a)$. We need to show that there is a larger $\gamma$ such that $L_\gamma \models \varphi(\vec a)$. In order to do this, let's work in $L$ and first see that the following is necessarily forceable:
	\begin{equation} \label{eqn:ForcesHmodelsPhiAndUnbdedlyCardsInL} H_{\omega_1} \models (\varphi^L(\vec a) \land \text{ ``there are unboundedly many cardinals in $L$"}). \end{equation}
This holds since otherwise it is forceably necessary that $H_{\omega_1} \models \neg \varphi^L(\vec a)$, so $H^V_{\omega_1} \models \neg \varphi^L(\vec a)$ holds, a contradiction.

So, given some $\theta > \kappa$, we may force over $L$ to collapse $\theta$ to $\omega$. Then as (\ref{eqn:ForcesHmodelsPhiAndUnbdedlyCardsInL}) is necessarily forceable, there is further forcing to reach a model $V[G][H]$ such that 
	$H_{\omega_1}^{V[G][H]} \models \varphi^L(\vec a).$
Thus in $V[G][H]$, $\varphi^{L_{\omega_1}}(\vec a)$ holds. Since $\omega_1^{V[G][H]} = \gamma>\theta>\kappa$ in this extension now, and furthermore by (\ref{eqn:ForcesHmodelsPhiAndUnbdedlyCardsInL})
	$$L_{\gamma} \models \varphi(\vec a) \land \text{ ``there are unboundedly many cardinals"},$$ 
we now have a suitable $\gamma$ that is inaccessible in $L$ and $L_{\gamma} \models \varphi(\vec a)$ as desired.
\end{proof}

\begin{observation}
The following consistency results hold.
\begin{enumerate}
	\item \label{item:cccL} $\bflMPccc \implies \c^V$ is locally uplifting in $L$.
	\item \label{item:pL} $\bflMPp \implies \aleph_2^V$ is locally uplifting in $L$.
	\item $\bflMPc \implies \aleph_2^V$ is locally uplifting in $L$.
	\item \label{item:scL} $\bflMPsc \implies \aleph_2^V$ is locally uplifting in $L$.
\end{enumerate}
\end{observation}
\begin{proof}
(\ref{item:pL})-(\ref{item:scL}) hold by collapsing to $\omega_1$ instead of $\omega$ in the proof of \ref{theorem:bflMPinL}. For (\ref{item:cccL}) instead of collapsing $\kappa$ to be as small as desired, blow up the continuum as needed, like in the comparable proof for $\bfMPccc$ in \cite[Thm.~31.2]{Hamkins:2003jk}.
\end{proof}
	
\begin{theorem} \label{theorem:forcingbflMP} If $\delta$ is locally uplifting, then there is a forcing extension in which $\bflMP$ holds and $\delta = \aleph_1$. \end{theorem} 
\begin{proof}
Let $\delta$ be locally uplifting.
Define the $\delta$-length lottery sum finite support iteration $\P = \P_\delta$ as follows: for $\alpha < \delta$ let 
	$\P_{\alpha+1} = \P_\alpha *\dot{\Q}_\alpha$
where $\dot{\Q}_\alpha$ is a $\P_\alpha$-name for the lottery sum of all minimal rank posets that force some sentence relativized to the $\HC$ of $V_\delta^{\P_\alpha}$ to be necessary. In particular, let $\Phi$ be the collection of sentences $\varphi(\vec a)$ where, in $V_\delta$, $\vec a$ is a $\P_\alpha$-name for an element of $\HC$ such that $V_\delta$ models that $\varphi^\HC(\vec a)$ is forceably necessary.
So $\Phi$ is the set of all possible ``local buttons" available at this point in the iteration. Then we let 
	$$\dot{\Q}_\alpha = \bigoplus_{\varphi \in \Phi} \set{ \dot {\Q} \in V_\delta^{\P_\alpha} }{ \dot \Q \text{ is least rank,  $V_\delta^{\P_\alpha} \models$ ``$\dot{\Q}$ forces `$\varphi(\vec a)^{\HC}$ is necessary.'"} }$$ 
	We shall refer to this definition as the least-rank $\bflMP$ lottery sum iteration of length $\delta$. 
	
Since we will want the full iteration $\P$ to remain relatively small in size and to have the $\delta$-$cc$, notice that here we insist that the parameters for our sentences come from $H_{\omega_1}^{V_\delta^{\P_\alpha}}$. As $\delta$ is inaccessible, it is large enough so that $H_{\omega_1}^V = H_{\omega_1}^{V_\delta}$, and moreover this remains true in each subsequent extension in the iteration so $H_{\omega_1}$ in the subsequent extensions gets interpreted the same in $V_\delta^{\P_\alpha}$ as in $V^{\P_\alpha}$. This is because $\delta$ is locally uplifting: if $\dot \Q' \in V^{\P_\alpha}$ forces that a sentence $\varphi(\vec a)^{H_{\omega_1}}$ is necessary, where $\vec a \in H_{\omega_1}$, then take $\theta$ large enough so that $\dot \Q' \in V^{\P_\alpha}_\theta$. Then as $\delta$ is locally uplifting we have an inaccessible $\gamma > \theta$ such that 
	$$V^{\P_\alpha}_\gamma \models \text{``There is a forcing notion $\dot \Q$ which forces `$\varphi(\vec a)^{H_{\omega_1}}$ is necessary.'"}$$ 
and
	$V^{\P_\alpha}_\delta$ models the above sentence as well.
So since each of the iterands of the forcing $\P$ are taken to be of least rank, they are all in $V_\delta$ anyway. If on the other hand at stage $\alpha+1$ we have that $\dot \Q' \in V_\delta^{\P_\alpha}$ is of least rank forcing that a sentence $\varphi(\vec a)^{H_{\omega_1}}$ is necessary, then the only way it could be wrong is that there is some further forcing $\dot \R'$ that is not in $V_\delta^{\P_\alpha*\dot \Q'}$ that forces the sentence to be false. But then we may take $\gamma$ larger than the verification of this forcing $\dot \R'$, and use the fact that $\delta$ is locally uplifting to see that $$V^{\P_\alpha*\dot \Q'}_\delta \models ``\text{There is a forcing notion $\dot \R$ which forces $\neg \varphi(\vec a)^{H_{\omega_1}}$."}$$ This contradicts the choice of $\dot \Q'$ in $V_\delta$, which means that $V_\delta$ is correct.
Thus the iteration is the same as if it were defined over $V$.


Now suppose that $G \subseteq \P$ is generic over $V$. Let's see that $V[G] \models \bflMP$. Assume toward a contradiction that it fails: namely $\varphi(\vec a)$ is a sentence with $\vec a \in H^{V[G]}_{\omega_1}$ such that in $V[G]$ it is a local button, and also that
$\varphi(\vec a)^{H_{\omega_1}}$ is not true in $V[G]$. Let us also take $p \in G$ forcing this to be the case.

Note that $H_\delta^{V[G]} = H_{\omega_1}^{V[G]}$, since $\delta$ is regular and $\P$ has the $\delta$-cc, so the length of the iteration is collapsed to $\omega_1$.		

Let $\dot{\Q}$ be a name for $\Q$, a least rank poset in $V^{\P}$ and $\dot{\vec a}$ be a name for $\vec a$ such that in $V^\P$, we have that
``$\varphi(\dot{\vec a})^{H_{\omega_1}}$ is necessary."

Since $\P$ has the $\delta$-$cc$, at no stage in the iteration could $\delta$ be collapsed. This means that there is some stage where the parameters $\dot{\vec a}$ appear. Thus we may find a stage in the iteration where the parameters $\vec a$ are available, past the support of $p$, say $\vec a \in V_\delta[G_\alpha]$. 

Now we let $\theta$ satisfy $\P \in V_\theta$ and $\dot \Q \in V_\theta^\P$. Then as $\delta$ is locally uplifting, we have that there is an inaccessible $\gamma >\theta$ satisfying 
	$$V_\gamma[G_\alpha] \models ``\varphi(\vec a)^{H_{\omega_1}} \text{ is forceably necessary."}$$
Namely, $\Ptail * \dot \Q$ makes $\varphi(\vec a)^{H_{\omega_1}}$ necessary.  So by the fact that $\delta$ is uplifting, we have that 
	$$V_\delta[G_\alpha] \models ``\varphi(\vec a)^{H_{\omega_1}} \text{ is forceably necessary."}$$
Moreover, $\varphi(\vec a)^{H_{\omega_1}}$ must continue to be a local button in later stages, since it is a local button in $V[G]$. So it is dense for the local button to be pushed -- $\varphi(\vec a)^{H_{\omega_1}}$ is necessary in $V_\delta[G_\beta]$ for some $\beta<\delta$. Thus $\varphi(\vec a)^{H_{\omega_1}}$ is true in $V[G]$, after the rest of the iteration; a contradiction.
\end{proof}

\begin{theorem} \label{thm:lMPGamma} Let $\delta$ be a locally uplifting cardinal. Then there are forcing extensions in which we have the following:
\begin{enumerate}
	\item\label{item:lproper} $\bflMPp + \; \delta= \c = \aleph_2$. 
	\item\label{item:lctblyclosed} $\bflMPc + \; \delta=\aleph_2 + \; \CH$. 
	\item\label{item:lsubcomplete} $\bflMPsc + \; \delta=\aleph_2+ \; \CH$. 
\end{enumerate}	
\end{theorem}
\begin{proof}
For (\ref{item:lproper})-(\ref{item:lsubcomplete}), follow the same blueprint as (\ref{theorem:forcingbflMP}), by defining a least-rank $\bflMP_\Gamma$ lottery sum iteration of length $\delta$. Use $H_{\omega_2}$ instead of $H_{\omega_1}$ and relativize to the particular forcing class, modifying the support of the iteration as in Theorem \ref{thm:RA+MP} for each forcing class. 
%
\end{proof}

Again more should be said about the case of $ccc$ forcing, since $ccc$ forcing notions are not closed under lottery sums. 

\begin{theorem} Let $\delta$ be a locally uplifting cardinal. Then there is a forcing extension in which $\bflMPccc$ holds and $\delta = \c$.
\end{theorem}
\begin{proof}
	Let $\delta$ be locally uplifting. We need to enumerate all formulas in the language of set theory, together with elements of $H_\delta$, with unbounded repetition. Let $f$ be such a function taking each $\alpha <\delta$ to a formula $\varphi(x)$ together with a parameter set $\vec a \in H_\delta$. We will use this enumeration as a kind of book-keeping function for the iteration. 
	
	Define the $\delta$-length lottery sum finite-support iteration $\P = \P_\delta$ all of whose initial segments will be elements of $V_\delta$. At stage $\alpha$ in the iteration, look at $f(\alpha) = \langle \varphi(x), \vec a \rangle$. Check as to whether in $V_\delta$, $\vec a$ is a $\P_\alpha$-name for an element of $H_\c$ and whether $\varphi(\vec a)$ is a local $ccc$-button; namely whether:
	$$ \text{ $V_\delta^{\P_\alpha} \models $ ``$\varphi^{H_\c}(\vec a)$ is $ccc$-forceably $ccc$-necessary."} $$
	If so, then 
	$\P_{\alpha+1} = \P_\alpha *\dot{\Q}_\alpha$
where $\dot{\Q}_\alpha$ is a $\P_\alpha$-name for a $ccc$ poset that forces $\varphi^{H_\c}(\vec a)$ to be $ccc$-necessary.

We can use the same reasoning as given in the proof of (\ref{theorem:forcingbflMP}) to see that since $\delta$ is locally uplifting, the iteration is the same as if it were defined over $V$, and that in this case $\P$ is $ccc$. 

Let $G \subseteq \P$ be generic over $V$. In $V[G]$ we have that $\delta=\c$ since $\P$ is $ccc$, $\delta = |\P|$ is inaccessible, and $\delta$-many reals are added. To see that $\bflMPccc$ holds in $V[G]$, let $\varphi(\vec a)$ be a local $ccc$-button with $\vec a \in H_{\c}^{V[G]}$. Since $\P$ is $ccc$, there is some stage $\alpha<\delta$ in the iteration where $\vec a$ starts to appear as a $\P_\alpha$-name in the iteration. Since the book-keeping function has unbounded repetition, we may use the locally uplifting cardinal as in the proof of (\ref{theorem:forcingbflMP}) to see that it is dense for $\varphi^{H_\c}(\vec a)$ to be pushed. This makes $\varphi^{H_\c}(\vec a)$ $ccc$-necessary in $V_\delta[G_\beta]$ for some $\alpha \leq \beta<\delta$. Thus $\varphi(\vec a)^{H_{\c}}$ is true in $V[G]$, after the rest of the iteration, which is $ccc$, as desired.
\end{proof}

\bibliographystyle{alpha}
\bibliography{../ApplicationStuff/CV-RS-Pubs-Refs-extra/BIB}
\end{document}